\documentclass[10pt]{article}
\usepackage{geometry} 
\usepackage{amsmath, amssymb, amsthm} 
\usepackage{mathtools} 
\usepackage{mathrsfs}  
\usepackage{authblk}   
\usepackage{comment}   
\usepackage{xcolor}   
\usepackage{titling}  
\usepackage[hidelinks]{hyperref} 
\hypersetup{colorlinks=true,
	linkcolor=blue,
	urlcolor=blue,
	anchorcolor=blue,
	citecolor=blue}
\pretitle{\vspace{-0.5em}\begin{center}\LARGE}  
	\posttitle{\end{center}\vspace{-0.5em}}    	
\preauthor{\vspace{0em}\begin{center}\large} 
	\postauthor{\end{center}\vspace{-0.5em}}    
\predate{\vspace{0em}\begin{center}}         
	\postdate{\end{center}\vspace{0.5em}}         
    
\numberwithin{equation}{section} 
\allowdisplaybreaks[4] 

\theoremstyle{plain}            
\newtheorem{thm}{Theorem}[section]
\newtheorem{lem}[thm]{Lemma}
\newtheorem{prop}[thm]{Proposition}

\theoremstyle{definition}       
\newtheorem{defn}[thm]{Definition}

\theoremstyle{remark}           

\newcommand{\qedd}{\hfill $\square$} 
\newcommand{\step}[1]{%
	\par\noindent 
	\textit{Step}~#1.\enspace 
}

\newcounter{hyp}[section]
\renewcommand{\thehyp}{\textbf{(H\arabic{hyp})}}
\newcommand{\hyp}[2][]{%
	\refstepcounter{hyp}%
	\noindent\thehyp\ifx\\#1\\\else\ (#1)\fi\, #2%
}
\newcounter{as}[section]
\renewcommand{\theas}{\textbf{(A\arabic{as})}}
\newcommand{\as}[2][]{%
	\refstepcounter{as}%
	\noindent\theas\ifx\\#1\\\else\ (#1)\fi\, #2%
}
\newcommand{\hypref}[1]{\begingroup\hypersetup{linkcolor=black}\textbf{\hyperref[#1]{\ref*{#1}}}\endgroup}
\newcommand{\asref}[1]{\begingroup\hypersetup{linkcolor=black}\textbf{\hyperref[#1]{\ref*{#1}}}\endgroup}

\newcommand{\R}{\mathbb{R}} 
\newcommand{\N}{\mathbb{N}} 
\def\P{\mathbb{P}} 			
\newcommand{\E}{\mathbb{E}} 

\newcommand{\var}{\operatorname{Var}}        

\title{Reflected stochastic partial differential equations with fully local monotone coefficients in infinite dimensional domains}
\author{Qi Li$^{1}$, Yue Li$^{2}$, Tusheng Zhang$^{1,3}$}
\date{\small\today}
\begin{document}
	\maketitle
    \addtolength{\skip\footins}{-1mm} 
    {\renewcommand{\thefootnote}{1}  
		\footnotetext{               
			School of Mathematical Sciences, University of Science and Technology of China, Hefei, Anhui, China.
			\textit{Email:}\texttt{\href{mailto:vivien777@mail.ustc.edu.cn}{vivien777@mail.ustc.edu.cn}}
		}
    }
    {\renewcommand{\thefootnote}{2}  
    	\footnotetext{               
    		School of Mathematics and Statistics, Nanjing University of Science and Technology, Nanjing, Jiangsu, China.
    		\textit{Email:} \texttt{\href{mailto:yueli@njust.edu.cn}{yueli@njust.edu.cn} 
            \textnormal{(Corresponding author)}
    	}
    }
    }
    {\renewcommand{\thefootnote}{3}  
		\footnotetext{               
    		Department of Mathematics, University of Manchester, Oxford road, Manchester M13 9PL, United Kingdom
            \textit{Email:}\texttt{\href{mailto:tusheng.zhang@manchester.ac.uk}{tusheng.zhang@manchester.ac.uk}}   
		}
	}
	\vspace{-1cm}
    \begin{abstract}

        This paper establishes the well-posedness of stochastic partial differential equations with reflection in an infinite-dimensional ball, within the fully local monotone framework. Our result is very general, including many important models such as the stochastic Allen-Cahn equations, stochastic $p$-Laplacian equations and stochastic 3D tamed Navier-Stokes equations, as well as more complex systems like the stochastic Cahn-Hilliard equations and stochastic 2D liquid crystal models. The approach relies on the penalization method, pseudo-monotonicity techniques and Mazur’s lemma.
    	 
        \vskip0.3cm
    	\noindent{\bfseries Keywords:}{Stochatic partial differential equations with reflection, fully local monotone, random measure,  weak topology, stochastic 3D tamed  Navier-Stokes equations}\\
    	\noindent{\bfseries MSC 2020:} 60H15; 60J55; 35R60.
        \end{abstract}

 \section{Introduction}
Let $ H $ be a separable Hilbert space with inner product $(\cdot,\cdot)$ and norm $ |\cdot|_{H} $. Let $ V $ be a reflexive Banach space that is continuously and densely embedded into $ H $. The norms of $ V $ and its dual space $ V^{*} $ are denoted by $ \|\cdot\|_{V} $ and $ \|\cdot\|_{V^{*}} $ respectively. If we identify the Hilbert space $ H $ with its dual space $ H^{*} $ by the Riesz representation, then we obtain a Gelfand triple
\[
V\subseteq H\subseteq V^{*}.
\]
We denote by $ \langle f,v\rangle $ the dual pairing between $ f\in V^{*} $ and $ v\in V $. It is easy to see that
\[
(u,v)=\langle u,v\rangle,\quad\forall\,u\in H,v\in V. 
\]
Let $ W $ be a cylindrical Wiener process on another separable Hilbert space $ U $ defined on some probability space $ (\Omega,\mathcal{F},\mathbb{P}) $ with normal filtration $ \mathbb{F} $.

Let $T>0 $ be fixed in this paper. Consider the following stochastic partial differential equations (SPDEs) with reflection:
\begin{equation}\label{SEE}
	\left\{
	\begin{aligned}
		dX(t)&=A(t,X(t))dt +B(t,X(t))dW(t) +dL(t), t\in(0,T],\\
		X(0)& =X_0, \quad X_0\in \overline{D},
	\end{aligned}
	\right.
\end{equation}
where $X$ is a $\overline{D}$-valued continuous stochastic process and $L$ is a local time, which is an $H$-valued stochastic process with locally bounded variation. Here $D:=B(0,1)$ denotes the open unit ball in $H$ and $\overline{D}$ its closure. 
The mappings 
\[
A:[0,T]\times V\to V^{*},\quad B:[0,T]\times V\to L_{2}(U,H)
\]
are progressively measurable, where $ L_{2}(U,H) $ is the space of Hilbert-Schmidt operators from $ U $ to $ H $ with the norm denoted by $ \|\cdot\|_{L_{2}} $.

The following is the definition of solutions to Problem \eqref{SEE}.
\begin{defn}\label{solutiondef}
	A pair $(X,L)$ is said to be a solution of the reflected problem $(\ref{SEE})$ iff the following conditions are satisfied:
	\begin{itemize}
		\item[(i)] $X$ is a $\overline{D}$-valued continuous and $\mathbb{F}$-progressively measurable stochastic process with $X\in L^\alpha ([0,T];V)$, $\mathbb{P}$-a.s. for some $\alpha \in (1,\infty)$.  The corresponding $V$-valued process is strongly $\mathbb{F}$-progressively measurable;
		\item[(ii)] $L$ is an $H$-valued, $\mathbb{F}$-progressively measurable stochastic process of locally bounded  variation such that $L(0)=0$ and
		\begin{equation*}
			\mathbb{E}[|\text{Var}_H(L)([0,T])|^2]<+ \infty,
		\end{equation*}
		where, for a function $v:[0,\infty)\rightarrow H$, Var$_H(v)([0,T])$ is the total variation of $v$ on $[0,T]$ defined by
		\begin{equation*}
			\text{Var}_H(v)([0,T]):= \sup\sum_{i=1}^n|v(t_i)-v(t_{i-1})|_H,
		\end{equation*}
		where the supremum is taken over all partitions $0 = t_0 < t_1 <\cdots < t_{n-1} < t_n = T$, $n\in \mathbb{N}$, of the interval $[0,T]$;
		\item[(iii)] $(X,L)$ satisfies the following integral equation in $V^{\ast}$, for every $t\in [0,T]$,
		\begin{equation}\nonumber
			X(t)=X_0+\int_0^t A(s,X(s))ds +\int_0^tB(s,X(s)))dW(s) +L(t),  \quad\mathbb{P}\text{-a.s.}
		\end{equation}
		\item[(iv)] for every $T>0$, for every $\phi \in C([0,T],\overline{D})$,
		\begin{equation} \label{def5}
			\int_0^T (\phi(t)-X(t),L(dt))\geq 0,\quad \mathbb{P}\text{-a.s.}
		\end{equation}
		where the integral on LHS is the Riemann-Stieltjes integral of the $H$-valued function $\phi-X$ with respect to an $H$-valued bounded-variation function $L$.
	\end{itemize}
\end{defn}

The analysis of random field dynamics constrained to stay in a prescribed domain naturally leads to stochastic partial differential equations with reflection. Such equations arise crucially in modeling physical systems with hard boundaries, such as evolving interfaces near walls or confined stochastic fluids. As an example, T. Funaki and S. Olla  \cite{FO01} proved that the fluctuations of a $\nabla \phi$ interface model near a hard wall converge in law to the stationary solution of a SPDE with reflection. 

\vspace{1em}
The study of reflected SPDEs driven by space-time white noise was initiated by  Nualart and Pardoux \cite{NP92} for additive noise case, extended to general  diffusion coefficient $\sigma$ but without uniqueness by Donati-Martin and Pardoux \cite{DP93}, completed with uniqueness for  general $\sigma$ by T. Xu and the third author  \cite{XZ09}. See \cite{DZ07} for reflected stochastic Cahn-Hilliard equations. Various properties of the solution of the real-valued SPDEs with reflection were studied in \cite{DMZ06,DP97,HP89,Za01,Z10}. 

\vspace{1em}
While reflection problems are well studied in finite dimensions, their extension to infinite-dimensional domains requires new analytical frameworks. The reflection problem for 2D stochastic Navier-Stokes equations with periodical boundary conditions in an infinite-dimensional ball was studied by  V. Barbu, G. Da Prato and L. Tubaro \cite{BDT11}, using stochastic variational inequality, Galerkin approximations and the Kolmogorov equations. They then studied the reflected problem for stochastic evolution equations driven by additive noise on a closed convex subset in a Hilbert space in \cite{BDT12}. Recently Brze\'{z}niak and the third author \cite{BZ23} studied the case of multiplicative noise for stochastic evolution equations via a direct stochastic penalization, whose framework successfully encompasses important models like the 2D stochastic Navier-Stokes equations. For its large deviation and ergodic properties, see \cite{BLZ24} and \cite{BLZ}.  Inspired by \cite{BZ23} but aiming for a more general setup, we investigate the reflection problem for SPDEs within the recently developed setting of fully local monotone framework in \cite{RSZ24}.

\vspace{1em}
The theory of monotone operators, initiated by Minty \cite{M62} and further developed by Browder, Leray, Lions, and others (see \cite{Li69,Ze90}), provides a powerful framework for analyzing nonlinear partial differential equations. For SPDEs, the variational approach was pioneered by Pardoux \cite{P72,P75}, and then Krylov and Rozovskii\cite{KR79} and Gy\"{o}ngy \cite{G82}. This theory was substantially extended by Liu and R\"{o}ckner \cite{LR10}, who introduced the classical local monotonicity framework. More recently, R\"{o}ckner, Shang and the third author established the well-posedness theory for SPDEs with fully local monotone coefficients \cite{RSZ24}, achieving a framework of remarkable generality.

\vspace{1em}
The aim of this paper is to establish the well-posedness for SPDEs with reflection in an infinite-dimensional ball, within the fully local monotone framework. Our approach builds upon the penalization method as in \cite{BZ23}.  In our framework, however, the operator $A$ lacks continuity, and the approximate solutions $X^n$ do not converge strongly in $L^\alpha(0,T;V^*)$ in probability. As a consequence, the penalty term $L^n$ fails to converge in the strong norm of $C([0,T];V^*)$ in probability, in contrast to the results in \cite{BZ23}. To overcome this difficulty, using the fact that the Bochner integral operator $I:L^1([0,T]\times \Omega;V^*)\to L^1(\Omega;C([0,T];V^*))$ defined by
$$(If)(t,\omega):=\int_0^t f(s,\omega)ds,\quad f\in L^1([0,T]\times \Omega;V^*),$$
is bounded and linear, hence weakly continuous, we pass the weak convergence of $A(\cdot,X^n)$  to $\int_0^\cdot A(s,X^n(s))ds$. Consequently, $L^n$ converge weakly in $L^1(\Omega;C([0,T];V^*))$. Then we apply Mazur’s lemma to extract convex combinations of $L^n$ that converge almost surely in $C([0,T];V^*)$, recovering the necessary strong convergence.
A second issue arises in identifying limit of $A(\cdot,X^n)$ via the monotonicity method, as the standard technique of comparing $\E|X^n(t)|_H^2$ and $\E|X(t)|_H^2$ requires $X$ to be an It\^{o} process, which is unknown in advance at in our case. We avoid this by using the  Cauchy  property of  $\{X^n\}$ in $L^2(\Omega;C([0,T];H))$. Together, these arguments allow us to extend the penalization method to a broader class of monotone operators without continuity and under weaker convergence assumptions.

\vspace{1em}
Our work under fully local monotone framework provides a general approach for studying reflected SPDEs covering not only the models in \cite{BDT12,BZ23}, but also a broader class, including stochastic 2D Navier-Stokes equations, porous media equations, reaction-diffusion equations, fast-diffusion equations, $p$-Laplacian equations, Burgers equations, Allen-Cahn equations, 3D Leray-$\alpha$ model, 2D Boussinesq system, 2D magneto-hydrodynamic equations, 2D Boussinesq model for the B\'{e}nard convection, 2D magnetic B\'{e}nard equations, some shell models of turbulence (GOY, Sabra, dyadic), power law fluids, the Ladyzhenskaya model, the Kuramoto-Sivashinsky equations and the 3D tamed Navier-Stokes equations, 3D tamed Navier-Stokes equations, some quasilinear PDEs, Cahn-Hilliard equations, liquid crystal models and Allen-Cahn-Navier-Stokes systems, see e.g. \cite{LR10,LR15,RSZ24}.

\vspace{1em}
The paper is organized as follows: Section 2 introduces the  framework and the fully local monotonicity assumptions. Section 3 is devoted to the analysis of the penalized equations and the derivation of crucial a priori estimates. Section 4 contains our main result and proofs, establishing the existence and uniqueness of solutions. Section 5 presents the applications of our result to some models.

\vspace{1em}
Conventions on constants. Throughout the paper, $C$ denotes a generic positive constant whose value may change from line to line. The dependence of constants on parameters if needed will be indicated, e.g. $C(T)$.

\section{Framework}

We introduce the following conditions on the coefficients $A$ and $B$. Let $C_0>0$ and $\alpha\in(1,\infty)$.

\hyp{(Hemicontinuity) for a.e. $t\in[0,T]$, the mapping $\mathbb{R}\ni\lambda\longmapsto\langle A(t,u+\lambda v),x\rangle\in\mathbb{R}$ is continuous for any $u,v,x\in V$.
}\label{H1}

\hyp{(Local monotonicity) there exist nonnegative constants $\gamma$ and $C$ such that for a.e. $t\in[0,T]$ and any $u,v\in V$,
	\begin{align*}
		2\langle A(t,u)-A(t,v),u-v\rangle + \|B(t,u)-B(t,v)\|_{L_{2}}^{2} \leq \left[C_0+\rho(u)+\eta(v)\right]|u-v|_{H}^{2},
	\end{align*}
	where $\rho$ and $\eta$ are two measurable functions from $V$ to $\mathbb{R}$ satisfying
	$$
	|\rho(u)|+|\eta(u)| \leq C\left(1+\|u\|_{V}^{\alpha}\right)\left(1+|u|_{H}^{\gamma}\right),
	$$}\label{H2}
\!\hyp{(Coercivity) there exists a constant $c>0$ such that for a.e. $t\in[0,T]$ and any $u\in V$,
	\[
	2\langle A(t,u),u\rangle + \|B(t,u)\|_{L_{2}}^{2} \leq C_0\left(1+|u|_{H}^{2}\right) - c\|u\|_{V}^{\alpha}.
	\]     
} \label{H3}
\!\hyp{(Growth) there exist nonnegative constants $\beta$ and $C$ such that for a.e. $t\in[0,T]$ and any $u\in V$,
	\[
	\|A(t,u)\|^{\frac{\alpha}{\alpha-1}}_{V^{*}} \leq C\left(1+\|u\|_{V}^{\alpha}\right)\left(1+|u|_{H}^{\beta}\right).
	\]
}\label{H4}
\!\hyp{For a.e. $t\in[0,T]$ and $u,v$ in $V$,
	\begin{align}
		\|B(t,u)-B(t,v)\|_{L_{2}}^2  &\leq C_0|u-v|_{H}^{2}.\label{lip}\\
		\|B(t,u)\|_{L_{2}}^{2} &\leq C_0\left(1+|u|_{H}^{2}\right).
	\end{align}
} \label{H5}	
\vspace{-2em}

\section{The existence and the uniqueness of solutions to an approximated problem}

Define the projection mapping $\pi: H \to \overline{D}$ onto the closed unit ball $\overline{D} = \{ x \in H : |x|_H \leq 1 \}$ by, for $y \in H$,
\begin{equation*}
\pi(y) =
\begin{cases}
	y, & \text{if } |y|_H \leq 1, \\
	\frac{y}{|y|_H}, & \text{if } |y|_H > 1.
\end{cases}
\end{equation*}  
Observe that $I-\pi$ can be seen as the gradient of the function $\phi$ defined by:
\begin{equation*}
\phi(y) = \frac{1}{2} |\operatorname{dist}(y,D)|^{2} =
\begin{cases}
	0, & \text{if } |y|_H \leq 1, \\
	\dfrac{1}{2}(|y|_H-1)^{2}, & \text{if } |y|_H > 1.
\end{cases}
\end{equation*}
In other words,
\begin{equation*}
\nabla \phi(y)=y-\pi(y), \quad y \in H.
\end{equation*}

The following lemma states some straightforward properties of the projection $\pi$ that will be used later.
\begin{lem}\label{lempi}
The mapping $\pi$ has the following properties.
\begin{enumerate}
	\item[(i)](Lipschitz continuity) 
	\begin{equation}\label{pi1}
		|\pi(x) - \pi(y)|_H \leq 2|x - y|_H, \quad x, y \in H.
	\end{equation}    
	\item[(ii)] For any $x \in H$,
	\begin{equation}
		\begin{aligned}\label{pi2}
			(\pi(x), x - \pi(x)) &= |x - \pi(x)|_H, \\
			(x, x - \pi(x)) &= |x|_H |x - \pi(x)|_H. 
		\end{aligned}
	\end{equation}
	
	\item[(iii)] (Variational inequality) for any $x \in H$ and $y \in \overline{D}$, 
	\begin{equation}\label{pi3}
		(x - y, x - \pi(x)) \geq 0.
	\end{equation}
\end{enumerate}
\end{lem}

\vspace{1em}
For every $n \in \mathbb{N}$, we consider the following penalized SPDE:
\begin{equation}\label{pSPDE}
X^{n}(t) = X_{0} +\int_{0}^{t} A(t,X^{n}(s)) \, ds + \int_{0}^{t} B(s,X^{n}(s)) \, dW(s) - n \int_{0}^{t} \left(X^{n}(s) - \pi(X^{n}(s))\right) \, ds. 
\end{equation}

The well-poesedness of the penalized equation \eqref{pSPDE} follows from \cite{RSZ24}.
\begin{thm}\label{thm3.2}
Suppose that the embedding $V\subseteq H$ is compact and \hypref{H1}-\hypref{H5} hold. Then for any initial value $X_0\in L^p(\Omega,\mathcal{F}_0,\P;H)$ for some $p
\geq 2$, there exists a probabilistically strong solution to equation \eqref{pSPDE}. Furthermore, the following moment estimate holds:
\begin{equation}
	\mathbb{E}\left\{\sup_{t\in[0,T]}|X^n(t)|_{H}^{p}\right\}+\mathbb{E}\left\{\left(\int_{0}^{T}\|X^n(s)\|^{\alpha}_{V}\,ds\right)^{\frac{p}{2}}\right\}<C(n,p,T). 
\end{equation}
\end{thm}

In the following, we give some estimates on $X^n$ whose proofs are similar to the corresponding results in  (c.f. \cite{BZ23}).  For readers' convenience, short proofs are provided in the Appendix. In the next section we will prove the convergence of the sequence $\{X^n\}$ and show that the limit is a solution of equation \eqref{SEE}.
\begin{lem}\label{lem3.4}
For $T>0$, there exist nonnegative constants $K_0=K_0(T)$ and $K_1=K_1(T)$ such that 
\begin{equation}\label{lem3.4-1}
	\sup_n \mathbb{E}\left[ \sup_{t\in[0,T]} |X^n(t)|_H^4 \right] \leq K_0,
\end{equation}
\begin{equation}\label{lem3.4-2}
	\mathbb{E}\left[ \int_0^T |X^n(t)|_H^2 \langle X^n(t), X^n(t) - \pi(X^n(t)) \rangle \, dt \right] \leq \frac{K_1}{n}, \quad n\in\mathbb{N}.
\end{equation}
\end{lem}

\begin{lem}\label{lem3.6}
For $T>0$, there exist nonnegative constants $M_{1}=M_{1}(T), M_{2}=M_{2}(T)$ and $M_{3}=M_{3}(T)$ such that
\begin{align}
	\label{lem3.6-1}
	\sup_{n}\mathbb{E}\left[\left(n\int_{0}^{T}\left|X^{n}(s)-\pi\left(X^{n}(s)\right)\right|_{H}ds\right)^{2}\right]&\leq M_{1},\\
	\label{lem3.6-2}\sup_{n}\E\left[n\int_{0}^{T}\left|X^{n}(s)-\pi\left(X^{n}(s)\right)\right|_{H}^{2}ds \right] &\leq  M_{2}\\
	\label{lem3.6-3}
	\sup_{n}\mathbb{E}\left[\int_{0}^{T}\left\|X^{n}(s)\right\|_V^{\alpha}ds\right]&\leq M_{3}.
\end{align}
\end{lem}

\begin{lem}\label{lem3.8}
For $T>0$,
\begin{equation}\label{lem3.8-1}
	\lim_{n\to\infty}\mathbb{E}\left[\sup_{t\in[0,T]}\left|X^{n}(t)-\pi\left(X^{n}(t)\right)\right|_{H}^{4}\right]=0.
\end{equation}
\end{lem}

\section{The existence and the uniqueness of solutions to the reflected problem}

The aim of this section is to prove the following main result of the paper.

\begin{thm}\label{result}
Suppose that the embedding $V\subseteq H$ is compact and \hypref{H1}-\hypref{H5} are satisfied. The reflected equation \eqref{SEE} with inital value $X_0\in\overline{D}$ admits a unique solution $(X, L)$ in the sense of Definition \ref{solutiondef} that satisfies, for $T > 0$,
\begin{equation}\label{thn4.1-1}
	\mathbb{E}\left[\sup_{t\in[0,T]}\left|X(t)\right|_{H}^{2}+\int_{0}^{T}\left\|X(t)\right\|^{\alpha}_V dt\right]<\infty.
\end{equation}
\end{thm}

\noindent\textit{Proof of Theorem \ref{result}.}
We will show that the sequence $\{X^{n},n\geq 1\}$ defined in \eqref{pSPDE} converges to a solution to equation \eqref{SEE}.

\vspace{1em}
\step{1} We show that $\{X^n\}$ as a sequence of $C([0,T];H)$-valued random variables converge in $L^2(\Omega)$.
\begin{lem}\label{lem4.3}
There exists an $H$-valued continuous adapted process $X$ such that
\begin{equation}\label{lem4.3-1}
	\lim_{n\to\infty}\mathbb{E}\left[\sup_{t\in[0,T]}\left|X^{n}(t)-X(t)\right|_{H}^{2}\right]=0.
\end{equation}
\end{lem}

\begin{proof}
Choose and fix natural numbers $m\geq n$. 
For $\lambda>0$ and $n\in \mathbb{N}$, define a process $f_{m,n}$ by the following formula
\begin{align*}
	\phi_{m,n}(t)&:=C_0+\rho(X^{n}(t))+\eta(X^m(t))\\
	f_{m,n}(t)&:=\exp\left\{-\lambda\int_{0}^{t}\phi_{m,n}(s)ds \right\},\quad t\geq 0.
\end{align*}
Note that $\phi_{m,n}(t)$ is locally integrable $\P$-a.s. in view of \hypref{H2}, \eqref{lem3.4-1} and \eqref{lem3.6-2}. Applying the It\^{o}'s formula we have that
\begin{equation}\label{lem4.3-2}
	\begin{aligned}
		&f_{m,n}(t)\left|X^{n}(t)-X^{m}(t)\right|_{H}^{2}=-\lambda\int_{0}^{t}f_{m,n}(s)\phi_{m,n}(s)\left|X^{n}(s)-X^{m}(s)\right|_{H}^{2}ds \\
		&\qquad+4\int_{0}^{t}f_{m,n}(s)\left\langle X^{n}(s)-X^{m}(s),A(s,(X^{n}(s))-A(s,X^{m}(s)\right\rangle\,ds \\
		&\qquad+2\int_{0}^{t}f_{m,n}(s)\left\langle X^{n}(s)-X^{m}(s),\left(B(s,X^{n}(s))-B(s,X^{m}(s)) \right) dW(s) \right\rangle\\
		&\qquad-2n\int_{0}^{t}f_{m,n}(s)\left\langle X^{n}(s)-X^{m}(s),X^{n}(s)-\pi\left(X^{n}(s)\right)\right\rangle ds \\
		&\qquad+2m\int_{0}^{t}f_{m,n}(s)\left\langle X^{n}(s)-X^{m}(s),X^{m}(s)-\pi\left(X^{m}(s)\right)\right\rangle ds \\
		&\qquad+2\int_{0}^{t}f_{m,n}(s)\|B(s,X^n(s)-B(s,X^m(s)\|_{L^2}^2\,ds\\
		&\qquad:=I_{1}^{m,n}(t)+I_{2}^{m,n}(t)+I_{3}^{m,n}(t)+I_{4}^{m,n}(t)+I_{5}^{m,n}(t)+I_{6}^{m,n}(t).
	\end{aligned}
\end{equation}

By \hypref{H2} we know that
\begin{equation*}
	\begin{aligned}
		I_{2}^{m,n}(t)+I_{6}^{m,n}(t)\leq \int_0^t f^{m,n}(s)(C_0+\rho(X^n(s))+\eta(X^m(s))|X^n(s)-X^m(s)|_H^2\, ds.
	\end{aligned}
\end{equation*}
This, when combined with $I_1^{m,n}$, is negative for large enough $\lambda$, due to the specific form of $\phi_{m,n}$.  Using Burkholder's inequality analogous to \eqref{lem3.4-6} and by \hypref{H5}, the Lipschitz property of $B$, we have
\begin{equation}\label{lem4.3-3}
	\begin{aligned}
		\E\left[\sup_{s\in[0,t]}|I_{3}^{m,n}(s)|\right]
		\leq & \frac{1}{2}\E\left[\sup_{s\in[0,t]} (f_{m,n}(s)|X^n(s)-X^m(s)|_H^2)\right]\\
		&+C\,\E\left[\int_0^t f_{m,n}(s)\|B(s,X^n(s))-B(s,X^m(s))\|_{L^2}^2\, ds\right]\\
		\leq & \frac{1}{2}\E\left[\sup_{s\in[0,t]} (f_{m,n}(s)|X^n(s)-X^m(s)|_H^2)\right]\\
		&+C\,\E\left[\int_0^t f_{m,n}(s) |X^n(s)-X^m(s)|_{H}^2\, ds\right].
	\end{aligned}
\end{equation}
As $\pi(X^{n}(s))\in\overline{D}$ and $\pi(X^{m}(s))\in\overline{D}$, it follows from \eqref{pi3} that $\langle X^{n}(s)-\pi(X^{m}(s)),X^{n}(s)-\pi(X^{n}(s))\rangle\geq0$ and $\langle X^{m}(s)-\pi(X^{n}(s)),X^{m}(s)-\pi(X^{m}(s))\rangle\geq0$. Hence,
\begin{align}\label{lem4.3-4}
	I_{4}^{m,n}(t)&=-2n\int_{0}^{t}f_{m,n}(s)\langle X^{n}(s)-\pi(X^{m}(s)),(X^{n}(s)-\pi(X^{n}(s))\rangle ds \nonumber\\
	&\quad+2n\int_{0}^{t}f_{m,n}(s)\langle X^{m}(s)-\pi(X^{m}(s)),(X^{n}(s)-\pi(X^{n}(s))\rangle ds \nonumber\\
	&\leq 2n\int_{0}^{t}f_{m,n}(s)\langle X^{m}(s)-\pi(X^{m}(s)),X^{n}(s)-\pi(X^{n}(s))\rangle ds \nonumber\\
	&\leq \left(2n\int_{0}^{t}\left|X^{n}(s)-\pi(X^{n}(s))\right|_{H}ds\right)\sup_{0\leq s\leq t}\left|X^{m}(s)-\pi(X^{m}(s))\right|_{H},
\end{align}
as $f_{m,n}(s)\leq1$. The case for $I_{5}^{m,n}$ is analogous, achieved by swapping the roles of $n$ and $m$.

Substituting \eqref{lem4.3-3}-\eqref{lem4.3-4} into \eqref{lem4.3-2}, choosing $\lambda$ large enough, as in the proof of Lemma \ref{lem3.8}, using Burkholder's and Hölder's inequalities, as well as \hypref{H5}, we obtain that
\begin{align*}
	&\mathbb{E}\left[\sup_{0\leq s\leq t}\left(f_{m,n}(s)\left|X^{n}(s)-X^{m}(s)\right|_{H}^{2}\right)\right] \\
	&\quad\leq\frac{1}{2}\mathbb{E}\left[\sup_{0\leq s\leq t}\left(f_{m,n}(s)\left|X^{n}(s)-X^{m}(s)\right|_{H}^{2}\right)\right]+C\,\mathbb{E}\left[\int_{0}^{t}f_{m,n}(s)\left|X^{n}(s)-X^{m}(s)\right|_{H}^{2}ds\right] \\
	&\quad\quad+C\left(\mathbb{E}\left[\left(2n\int_{0}^{t}\left|X^{n}(s)-\pi(X^{n}(s))\right|_{H}ds\right)^{2}\right]\right)^{\frac{1}{2}}\left(\mathbb{E}\left[\sup_{0\leq s\leq t}\left|X^{m}(s)-\pi(X^{m}(s))\right|_{H}^{2}\right]\right)^{\frac{1}{2}} \\
	&\quad\quad+C\left(\mathbb{E}\left[\left(2m\int_{0}^{t}\left|X^{m}(s)-\pi(X^{m}(s))\right|_{H}ds\right)^{2}\right]\right)^{\frac{1}{2}}\left(\mathbb{E}\left[\sup_{0\leq s\leq t}\left|X^{n}(s)-\pi(X^{n}(s))\right|_{H}^{2}\right]\right)^{\frac{1}{2}}.
\end{align*}
By Gronwall's lemma and also \eqref{lem3.6-1} in Lemma \ref{lem3.6}, we get
\begin{equation*}
	\begin{aligned}
		&\mathbb{E}\left[\sup_{s\in[0,T]}\left(f_{m,n}(s)\left|X^{n}(s)-X^{m}(s)\right|_{H}^{2}\right)\right] \\
		&\quad\leq C(M_{T})^{\frac{1}{2}}\left(\mathbb{E}\left[\sup_{s\in[0,T]}\left|X^{m}(s)-\pi\left(X^{m}(s)\right)\right|_{H}^{2}\right]\right)^{\frac{1}{2}} \\
		&\quad\quad+C(M_{T})^{\frac{1}{2}}\left(\mathbb{E}\left[\sup_{s\in[0, T]}\left|X^{n}(s)-\pi\left(X^{n}(s)\right)\right|_{H}^{2}\right]\right)^{\frac{1}{2}}.
	\end{aligned}
\end{equation*}
From Lemma \ref{lem3.8} it follows that 
\begin{equation}\label{lem4.3-5}
	\lim_{m,n\to\infty}\E\left[\sup_{s\in[0,T]} \left(f_{m,n}(s)|X^n(s)-X^m(s)|_H^2\right)\right]=0.
\end{equation}

Now we are to prove the sequence of $C([0,T];H)$-valued random variables $\{X^n, n \ge 1\}$ is a Cauchy sequence in probablity, that is,
\[
\lim_{n,m \to \infty} \mathbb{P}\left( \sup_{s \in [0,T]} |X^n(s) - X^m(s)|_H^2\ge \delta \right) = 0, \quad \forall \delta > 0.
\]
Let $$\Omega_M=\left\{\sup_{s\in[0,T]} |X^i(s)|_H\leq M,\int_0^T \|X^i(s)\|_V^\alpha \, ds \leq M,i=n,m\right\}.$$  
On $\Omega_M$, $f_{m,n}(s)\geq e^{-2\lambda[(1+M^\gamma)(T+M)+C_0]}$ uniformly in $s\in[0,T]$.
Given $\delta > 0$, for any $M > 0$ we have
\begin{align*}
	&\mathbb{P}\left( \sup_{s \in [0,T]} |X^n(s) - X^m(s)|_H^2  \ge \delta \right)  \leq \mathbb{P}\left( \sup_{s \in [0,T]} |X^n(s) - X^m(s)|_H^2\geq \delta; 1_{\Omega_M}\right) +\mathbb{P}\left(\Omega_M^c\right)\\
	&\qquad\qquad\leq \P\left(\sup_{s\in[0,T]} \left(f_{m,n}(s)|X^n(s)-X^m(s)|_H^2\right)\geq \delta e^{-2\lambda[(1+M^\gamma)(T+M)+C_0]}\right)\\
	&\qquad\qquad\quad
	+\frac{C}{M^4}\sup_n\E\left[\sup_{s\in[0,T]} |X^n(s)|^4_H\right]
	+\frac{C}{M}\sup_n\E\left[\int_0^T \|X^n(s)\|_V^\alpha\,ds\right]
\end{align*}
Letting $m,n\to \infty$, by  Chebyshev inequality and \eqref{lem4.3-5}, together with \eqref{lem3.4-1} in Lemma \ref{lem3.4} and \eqref{lem3.6-3} in Lemma \ref{lem3.6}, the RHS of the above inequality converges to $0$, thereby showing that $\{X^n\}$ is a Cauchy sequence in probability. With this and the estimate \eqref{lem3.4-1}, the convergence in $L^2(\Omega;C([0,T];H))$ follows form Vitali convergence theorem. Hence, the proof of Lemma \ref{lem4.3} is complete.
\end{proof}

Fatou's lemma and Lemma \ref{lem3.8} give that
\begin{equation}\label{4.7}
\mathbb{E}\left[\sup _{s\in[0, T]}\left|X(s)-\pi(X(s))\right|_{H}^{2}\right]\leq\lim _{n\rightarrow\infty}\mathbb{E}\left[\sup _{s\in[0, T]}\left|X^{n}(s)-\pi\left(X^{n}(s)\right)\right|_{H}^{2}\right]=0. 
\end{equation}
This implies that $\mathbb{P}$-a.s., for every $t\in[0,T]$, $X(t)=\pi(X(t))\in\overline{D}$.

\vspace{2em}
\step{2}
Based on Lemma \ref{lem4.3} and the Lipschitz property of $B$ in \hypref{H5}, keep in mind that we have the following strong convergences:
\begin{itemize}
\item[(i)] $X^n\to X$ in $L^2(\Omega;C([0,T];H))$. 
\item[(ii)] $\int_0^\cdot B(s,X^{n}(s))dW(s)\to\int_0^\cdot B(s,X(s))dW(s) \text{ in } L^2(\Omega;C([0,T];H)).$
\end{itemize}
Denote the weak convergence by ``$\rightharpoonup$''. 
From (\ref{lem3.6-3}), there exists a subsequence (still denoted by $n$) such that
\begin{itemize}
\item[(iii)] $X^{n}\rightharpoonup{X}$ in $L^\alpha ([0,T]\times\Omega;V)$.
\item[(iv)] $A(\cdot,X^{n})\rightharpoonup Y$ in $L^{\frac{\alpha}{\alpha-1}}([0,T]\times\Omega;V^\ast).$
\end{itemize}
Here (iv) follows from the uniform estimate: for $p\geq 2,T> 0$, there exists $C(T,p)\geq 0$ such that
\begin{equation*}
\sup_n\Big\{\mathbb{E}\sup_{t\in[0,T]}|X^n(t)|_H^p+\mathbb{E}\int_0^T|X^n(t)|_H^{p-2}\|X^n(t)\|_V^\alpha dt\Big\} \leq C(T,p).
\end{equation*}
which is obtained by applying It\^{o}'s formula to $|X^n(t)|^p_H$.
Taking $p\geq\beta+2$ and using \hypref{H4} yields
$$\sup_n\|A(\cdot,X^n)\|_{L^{\frac{\alpha}{\alpha-1}}([0,T]\times\Omega;V^\ast)}\leq C(T,\beta),$$
whence the weak convergence follows.

\vspace{2em}
To identify $Y=A(\cdot,X)$, we will use the pseudo-monotonicity property of the operator $A$.  Recall the defintion of pseudo-monotonicity operator.
\begin{defn}
An operator $A$ from $V$ to $V^\ast$ is said to be pseudo-monotone, if the following property holds: if $u_n$ converges weakly to $u$ in $V$ and
\begin{equation*}
	\liminf_{n\to\infty}\langle A(u_n),u_n-u\rangle\geq 0,
\end{equation*}
then
\begin{equation*}
	\limsup_{n\to\infty}\langle A(u_n),u_n-v\rangle \leq \langle A(u),u-v\rangle,\quad \forall v\in V.
\end{equation*}
\end{defn}
Recall the following results from \cite{RSZ24} Lemma 2.15 and 2.16.

\begin{prop}\label{pseu}
Assume \hypref{H1} and \hypref{H2} hold, the embedding $V\subseteq H$ is compact. Then $A(t,\cdot)$ is pseudo-monotone from $V$ to $V^\ast$ for a.e. $t\in [0,T].$
\end{prop}
\begin{prop}\label{prop4.5}
If 
\begin{align*}
	X^n\rightharpoonup &X \quad in\ L^\alpha([0,T]\times\Omega;V),\nonumber\\
	A(\cdot,X^n(\cdot))\rightharpoonup& Y\quad in\ L^{\frac{\alpha}{\alpha-1}}([0,T]\times\Omega;V^\ast),\nonumber
\end{align*}
\begin{equation}
	\liminf_{n\to\infty}\mathbb{E}\left[\int_0^T\langle A(t,X^n(t)),X^n(t)\rangle dt\right]\geq \mathbb{E}\left[\int_0^T \langle Y(t),X(t)\rangle dt\right],\label{prop4.5-1}
\end{equation}
then $Y(\cdot)=A(\cdot,X(\cdot)),\ dt\otimes\mathbb{P}$-a.e.
\end{prop}

To obtain $Y=A(\cdot,X)$, the key step is to prove \eqref{prop4.5-1}. Unlike the approach in \cite{LR15,RSZ24}, we cannot derive this inequality by directly comparing between $\E|X^{n}(t)|^2$ and $\E|X(t)|_H^2$ due to an additional reflection term arising  in our setting. Moreover, it currently remains unclear whether $X(t)$ is an It\^{o} process. To proceed, we instead apply It\^{o}'s formula to $|X^{n}(T)-X^m(T)|_H^2$ and use the convergence of $X^n$ in $L^2(\Omega;C([0,T];H))$. More precisely, by It\^{o}'s formula we have, after rearrangement,
$$
\begin{aligned}
&\mathbb{E}\left[\int_0^T\langle A(t,X^n(t))- A(t,X^m(t)),X^n(t)-X^m(t)\rangle dt\right]\\
&\qquad=\frac{1}{2}\mathbb{E}|X^n(T)-X^m(T)|^2_H
-\frac{1}{2}\mathbb{E}\left[\int_0^T\|B (t,X^n(t))- B(t,X^m(t))\|_{L_2}^2dt\right]\\
&\qquad\quad-\mathbb{E}\left[\int_0^T\langle (B(t,X^n(t))-B(t,X^m(t)))dW(t),X^n(t)-X^m(t)\rangle\right]\\
&\qquad\quad +n\,\mathbb{E}\left[\int_0^T\langle X^n(t)-\pi(X^n(t)),X^n(t)-X^m(t)\rangle dt\right]\\
&\qquad\quad -m\,\mathbb{E}\left[\int_0^T\langle X^m(t)-\pi(X^m(t)),X^n(t)-X^m(t)\rangle dt\right].\label{prop4.5-2}
\end{aligned}
$$
In view of (i) and (ii) in Step 2, the first three terms on the RHS converge to $0$ as $n,m\to\infty$. The third term can be written as
$$\begin{aligned}
&n\,\mathbb{E}\left[\int_0^T\langle X^n(t)-\pi(X^n(t)),X^n(t)-\pi(X^m(t))\rangle dt \right]\\
&\ + n\,\mathbb{E}\left[\int_0^T\langle X^n(t)-\pi(X^n(t)),\pi(X^m(t))-X^m(t)\rangle dt\right]
=:J_1^{m,n}+J_2^{m,n},
\end{aligned}$$
where $(\ref{pi3})$ implies $J_1^{m,n}\geq 0$, while (\ref{lem3.6-1}), \eqref{lem4.3-1} and  the Lipschitz property of $I-\pi$ implies
$$
\begin{aligned}
\lim_{n,m\to\infty}|J_2^{m,n}|\leq& C^{\frac{1}{2}}\lim_{m\to\infty}\left\{\mathbb{E} \left[\sup_{t\in[0,T]}|X^m(t)-\pi(X^m(t))|_H^2\right]\right\}^{\frac{1}{2}}=0.    
\end{aligned}$$
Similarly, we can obtain that $$\liminf_{n,m\to\infty}-m\mathbb{E}\left[\int_0^T\langle X^m(t)-\pi(X^m(t)),X^n(t)-X^m(t)\rangle dt\right]\geq 0.$$
Combining the above estimates we have 
\begin{equation}
\liminf_{m,n\to\infty}\mathbb{E}\left[\int_0^T\langle A(t,X^n(t))- A(t,X^m(t)),X^n(t)-X^m(t)\rangle dt\right] \geq 0.  
\end{equation}

$$\begin{aligned}
&\mathbb{E}\left[\int_0^T\langle A(t,X^n(t))- A(t,X^m(t)),X^n(t)-X^m(t)\rangle dt\right]\\
=&\,\mathbb{E}\left[\int_0^T\langle A(t,X^n(t)),X^n(t)\rangle dt\right] -\mathbb{E}\left[\int_0^T\langle A(t,X^n(t)),X^m(t)\rangle dt\right]\\
&\,-\mathbb{E}\left[\int_0^T\langle A(t,X^m(t)),X^n(t)\rangle dt\right]+\mathbb{E}\left[\int_0^T\langle A(t,X^m(t)),X^m(t)\rangle dt\right].
\end{aligned}$$
Recall that $A(\cdot,X^n(\cdot))\rightharpoonup Y$ in $L^{\frac{\alpha}{\alpha-1}}([0,T]\times\Omega;V^\ast)$ and  $X^n\rightharpoonup X$ in $L^{\alpha}([0,T]\times\Omega;V)$. Firt fixing $n$ and letting $m\to \infty$, then letting $n\to\infty$ and using the fact $\liminf_{n\to\infty}\liminf_{m\to\infty}a_{mn}\geq \liminf_{m.n\to\infty}a_{mn}$, we can get
\begin{equation*}
2 \liminf_{n\to\infty}\mathbb{E}\left[\int_0^T\langle A(t,X^n(t)),X^n(t)\rangle dt\right]-2\mathbb{E}\left[\int_0^T\langle Y(t),X(t)\rangle dt\right]\geq 0.
\end{equation*}
Hence \eqref{prop4.5-1} is obtained. Then by Proposition \ref{prop4.5},
$Y(\cdot)=A(\cdot,X(\cdot)),\ dt\otimes\mathbb{P} $-a.e.

\vspace{2em}
\step{3}
Define the sequence of $H$-valued adapted stochastic processes $\{L^n\}_{n\in\N}$ by
\begin{equation*}
L^n(t) = -n \int_0^t \left( X^n(s) - \pi\left( X^n(s) \right) \right) ds, \quad t \geq 0. 
\end{equation*}
According to \eqref{lem3.6-1} in Lemma \ref{lem3.6},
\begin{equation}\label{lem4.5-1}
\sup_n \mathbb{E}\left[ \operatorname{Var}_H(L^n)([0, T])^2 \right] = \sup_n \mathbb{E}\left[ \left( n \int_0^T \left| X^n(t) - \pi\left( X^n(t) \right) \right|_H dt \right)^2 \right] < \infty. 
\end{equation}

Now we define a process 
$$L(t):=X(t)-X_0-\int_0^t A(s,X(s))ds-\int_0^t B(s,X(s))dW(s),\quad t\in[0,T].$$
\begin{lem}\label{lem4.5}
The process $\{L(t),0\leq t\leq T\}$ is an $H$-valued adapted process of bounded variation and
\begin{equation}\label{lem4.5-2}
	\mathbb{E}\left[ \var_H(L)([0, T])^2 \right] < \infty.
\end{equation}
\end{lem}

\begin{proof}
Recall from the penalized equation \eqref{pSPDE} that 
$$L^n(t)=X^n(t)-X_0-\int_{0}^{t} A(t,X^{n}(s)) \, ds - \int_{0}^{t} B(s,X^{n}(s)) \, dW(s).$$
Since $X^n$ and $\int_0^\cdot B(s,X^{n}(s))dW(s)$ converge strongly in $L^2(\Omega;C([0,T];H))$, the convergence of $L^n$ to $L$ depends on how the term  $\int_0^\cdot A(s,X^n(s))ds$ converges.
It has been shown that $A(\cdot,X^n)$ converges weakly to $A(\cdot,X)$ in $L^{\frac{\alpha}{\alpha-1}}([0,T]\times \Omega;V^*)$, hence also weakly in $L^1([0,T]\times \Omega;V^*)$. We then consider the convergence of the corresponding integrals. 
Define the integral operator $I:L^1([0,T]\times \Omega;V^*)\to L^1(\Omega;C([0,T];V^*))$ by
$$(If)(t,\omega):=\int_0^t f(s,\omega)ds,\quad f\in L^1([0,T]\times \Omega;V^*),$$
which is a bounded linear operator. In fact, it is also bounded as a  linear operator from $L^p([0,T]\times \Omega;V^*)$ to $L^p(\Omega;C([0,T];V^*))$ for $1\leq p\leq \frac{\alpha}{\alpha-1}$. Thanks to the weak continuity of the bounded linear operator $I$, we have
\begin{equation*}
	\int_0^\cdot A(s,X^n(s))ds\rightharpoonup\int_0^\cdot A(s,X(s))ds\text{ in }L^1(\Omega;C([0,T];V^*)).
\end{equation*}
Since $X^n$ and $\int_0^\cdot B(s,X^{n}(s))dW(s)$ converge strongly in the same space, it follows that 
\begin{equation*}
	L^n\rightharpoonup L \text{ in }L^1(\Omega;C([0,T];V^*)).
\end{equation*}

By Mazur's lemma, there exists a convex combination 
$$\tilde{L}^m:=\sum_{k=m}^{N_m}\lambda_k^m L^k, \ \sum_{k=m}^{N_m}\lambda_k^m =1,\ \lambda_k^m>0,\ k=1,2,\cdots ,N_m,$$
such that
$\tilde{L}^m\to L$ in $L^1(\Omega;C([0,T];V^*))$ strongly. Hence, along a subsequence (still denoted by $m$), we have $\tilde{L}^m\to L$ in $C([0,T];V^*)$, $\P$-a.s. $\omega$. We estimate the variation of $\tilde{L}^m$. Applying the triangle inequality for the $H$-norm and Jensen’s inequality yields, for any $m\in \N$,
\begin{equation*}
	\mathbb{E}\left[\left(\var_H(\tilde{L}^m)([0,T])\right)^2\right]\leq \mathbb{E}\left[\left(\sum_{k=m}^{N_m}\lambda_k^m\var_H(L^k)([0,T])\right)^2\right] \leq \mathbb{E}\left[\sum_{k=m}^{N_m}\lambda_k^m\left(\var_H(L^k)([0,T])\right)^2\right].
\end{equation*}
The RHS is further controlled by $\sup_n \mathbb{E}\left[\left(\var_H(L^n)([0,T])\right)^2\right]$ and hence by \eqref{lem4.5-1} is uniformy bounded.
Therefore, $\{\tilde{L}^m\}_{m\in N}$ is a sequence of $H$-valued adapted processes of bounded variation and satisfies 
\begin{equation}
	\sup_m \mathbb{E}\left[\left(\var_H(\tilde{L}^m)([0,T])\right)^2\right]<\infty. \label{lem4.5-3}
\end{equation}

Note that the total variation functional
\begin{equation}
	\var_H(\cdot)([0,T]):C([0,T];V^*)\ni v\to \var_H(v)([0,T])\in [0,+\infty]\label{lem4.5-4}
\end{equation}
is lower semi-continuous (see, e.g., \cite{BZ23}). Since $\tilde{L}^m\to L$ in $C([0,T];V^*)$, $\P$-a.s., we have $L$ is an $H$-valued process of bounded variation and $\mathbb{P}$-a.s.,
\begin{equation*}
	\var_H(L)([0,T])\leq \liminf_{n\to\infty}\var_H(\tilde{L}^m)([0,T]).
\end{equation*}
Consequently, Fatou's lemma and \eqref{lem4.5-3} give that
\begin{equation*}
	\E\left[\var_H(L)([0,T])^2\right]\leq \sup_m \mathbb{E}\left[\left(\var_H(\tilde{L}^m)([0,T])\right)^2\right]<\infty.
\end{equation*}
\end{proof}

\step{4} We will show that $(X,L)$ is a solution of Definition \ref{solutiondef} to equation \eqref{SEE}. To this aim we will verify the variational inequality.
Let us choose and fix a function $\phi\in C([0,T],\overline{D})$. By \eqref{pi3} in Lemma \ref{lempi}, for every $n\in\mathbb{N}$, $$\left\langle X^{n}(t)-\phi(t),X^{n}(t)-\pi(X^{n}(t))\right\rangle\geq 0,$$ 
from which we deduce that $\P$-a.s. for every $n\in \N$,
\begin{equation}\label{uLconv-1}
\int_{0}^{T}\left(\phi(t)-X^n(t),L^n(dt)\right)=-n\int_{0}^{T}\left(\phi(t)-X^n(t),X^n(t)-\pi\left(X^n(t)\right)\right)dt\geq 0.
\end{equation}
This will imply that  
$$\int_0^T \left(\phi(t) - X(t), L(dt)\right) \geq 0,\ \P\text{-a.s.}$$  
provided we can show that $\P$-a.s. 
$$
\int_0^T \left(\phi(t) - X(t), L(dt)\right) = \lim_{m \to \infty}\sum_{k=m}^{N_m}\lambda_k^m \int_0^T (\phi(t) - X^k(t), L^k(dt)).
$$
Observe that  
$$
\begin{aligned}
&\sum_{k=m}^{N_m}\lambda_k^m \int_0^T (\phi(t) - X^k(t), L^k(dt)) - \int_0^T (\phi(t) - X(t), L(dt)) \\
&\ = \sum_{k=m}^{N_m}\lambda_k^m \left[\int_0^T (X(t) - X^k(t),L^k(dt)) + \left( \int_0^T (\phi(t) - X(t),L^k(dt)) - \int_0^T (\phi(t) - X(t), L(dt)) \right)\right] \\
&\ =: C_1^m + C_2^m, \quad m\in \mathbb{N}.
\end{aligned}
$$
In view of \eqref{lem4.3-1} and  \eqref{lem4.5-1}, we infer that 
\begin{equation*}
\E|C_1^m| \leq \sum_{k=m}^{N_m}\lambda_k^m\left\{\E\left[\sup_{t \in [0,T]} |X^k(t) - X(t)|_H^2 \right]\right\}^{1/2} \left\{\sup_k\E\left[\left(\text{Var}_H(L^k)([0,T])\right)^2\right] \right\}^{1/2}\underset{m\to\infty}{\longrightarrow}0.
\end{equation*}  
The proof of $C_2^m$ term involves the density argument analogous to that in \cite{BZ23}. We know that $\tilde{L}^m \to L$ in $C([0, T ]; V^*)$, $\P$-a.s. Let $v = \phi - X$, which belongs to  $C([0,T];H)$. By the density of $C([0,T];V)$ in $C([0,T]; H)$, for every  $\varepsilon > 0$, we can choose $ v_\varepsilon \in C([0,T], V) $  such that $ |v_\varepsilon|_{C([0,T];H)} = \sup_{t \in [0,T]} |\phi(t) - v(t)|_H < \varepsilon$. Then  $C_2^m$  can be bounded as 
\begin{equation*}
\begin{aligned}
	|C_2^m| &= \left| \int_0^T \left(v(t), \tilde{L}^m(dt)\right) - \int_0^T \left(v(t), L(dt)\right) \right|\\
	&\leq\left| \int_0^T \left(v(t) - v_\varepsilon(t),\tilde{L}^m(dt)\right)\right| + \left| \int_0^T \left(v(t) - v_\varepsilon(t), L(dt)\right) \right|
	\\
	&\quad + \left|\int_0^T \left(v_\varepsilon(t), \tilde{L}^m(dt)\right) - \int_0^T \left(v_\varepsilon(t), L(dt)\right) \right|
	.
\end{aligned}   
\end{equation*}
In view of the uniform bounds \eqref{lem4.5-3} and \eqref{lem4.5-2} on the total variation of  $\tilde{L}^m$ and $L$, the expectation of the first two terms on the RHS of  can be bounded by  $C \varepsilon^{\frac{1}{2}}$, while the expectation of the third term tends to zero as  $m\to \infty$. As $ \varepsilon$  is arbitrary, we conclude that  
$$\lim_{m \to \infty} \E|C_2^m| = 0.$$ 
Therefore, we have shown that $(X,L)$ is a solution to equation \eqref{SEE}.

\vspace{1em}
\noindent\textit{Proof of the uniqueness part of Theorem \ref{result}}

Let $(X,L)$ be the solution to the reflected SPDE (\ref{SEE}) constructed above. Let $(X',L')$ be another solution to the reflected SPDE (\ref{SEE}). Set $$\varphi(t):=\exp\bigg(-\int_0^t[C_0+\rho(X(r))+\eta(X'(r))]dr\bigg).$$
Then $\varphi$ is a continuous process of finite variation. By It\^{o}'s formula, \hypref{H2} and the fact that $X(t),X'(t)\in\overline{D}$, we have for any $t\in [0,T],$
\begin{align*}
&\varphi(t)|X(t)-X'(t)|_H^2\\
=&-\int_0^t\varphi(s)[C_0+\rho(X(s))+\eta(X'(s))]|X(s)-X'(s)|_H^2ds\\
&+2\int_0^t\varphi(s)\langle X(s)-X'(s),A(s,X(s))-A(s,X'(s))\rangle ds\\
&+2\int_0^t\varphi(s) (X(s)-X'(s),[B(s,X(s))-B(s,X'(s))]dW(s))\\
&+\int_0^t\varphi(s)\|B(s,X(s))-B(s,X'(s))\|^2_{L_2}ds\\
&+2\int_0^t\varphi(s)(X(s)-X'(s),L(ds))-2\int_0^t\varphi(s)(X(s)-X'(s),L'(ds))\\
\leq&\ 2\int_0^t\varphi(s) (X(s)-X'(s),[B(s,X(s))-B(s,X'(s))]dW(s)).
\end{align*}
Let $\{\sigma_l\}\nearrow\infty$ be a sequence of stopping times such that the local martingale in the above inequality is a martingale. Then taking the expectation on both sides of the above inequality, we get
\begin{equation*}
\mathbb{E}\Big[\varphi(t\wedge\sigma_l)|X(t\wedge\sigma_l)-X'(t\wedge\sigma_l)|^2_H\Big]=0.
\end{equation*}
Letting $l\to\infty$ and applying Fatou's lemma yield that for any $t\in[0,T],$
\begin{equation}\label{uni}
\mathbb{E}\Big[\varphi(t)|X(t)-X'(t)|^2_H\Big]=0.    
\end{equation}
By \hypref{H2} and the estimate (\ref{thn4.1-1}), we can see
\begin{equation*}
\int_0^T[C_0+\rho(X(r))+\eta(X'(r))]dr<\infty,\quad \mathbb{P}\text{-a.s.}    
\end{equation*}
According to the definition of $\varphi(t)$ and the above estimates, for any $t\in[0,T]$ we have $\varphi(t,\omega)>0$ for $\mathbb{P}$-a.s. $\omega\in \Omega.$ Hence (\ref{uni}) implies the pathwise uniqueness of solutions to equation (\ref{SEE}).

We have completed the proof of Theorem \ref{result}.
\qed

\vspace{2em}
In order to formulate our next result let us recall that $L$ is an $H$-valued process whose trajectories are of locally bounded variation. We introduce the following notation
\[
|L|(t)=\operatorname{Var}_H(L)([0,t]),\quad t\in[0,\infty).
\]
For each $\omega\in\Omega$, the function $\mathbb{R}_{+}\ni t\mapsto |L| (t)\in\mathbb{R}_{+}$ is increasing. Hence, one can associate it with a unique measure $m_{ |L| }$, called the Lebesgue-Stieltjes measure, usually denoted by $d|L|(t)$. The proof of the following result is close to that of Proposition 4.6 in \cite{BZ23}, with $L^n$ replaced by its convex combination.

\begin{prop}\label{prop4.6}
Let $(X,L)$ be the solution to equation \eqref{SEE}. Then $\mathbb{P}$-a.s. the measure $d|L|(t)$ is supported on the set
\begin{equation*}
	\left\{t\in[0,\infty):X(t)\in\partial D\right\}=\left\{t\in[0,\infty):|X(t)|_H
	=1\right\}.
\end{equation*}
\end{prop}

\section{Applications}

The results of this paper can be applied to establish the existence and uniqueness of reflected problems for many interesting stochastic nonlinear evolution equations, including the 2D Navier-Stokes equations, porous media equations, reaction-diffusion equations, fast-diffusion equations, $p$-Laplacian equations, Burgers equations, Allen-Cahn equations, 3D Leray-$\alpha$ model, 2D Boussinesq system, 2D magneto-hydrodynamic equations, 2D Boussinesq model for the B\'{e}nard convection, 2D magnetic B\'{e}nard equations, some shell models of turbulence (GOY, Sabra, dyadic), power law fluids, the Ladyzhenskaya model, the Kuramoto-Sivashinsky equations and the 3D tamed Navier-Stokes equations, see \cite{LR15}. Our results  are also applicable to some quasilinear PDEs, Cahn-Hilliard equations, liquid crystal models and Allen-Cahn-Navier-Stokes systems, see \cite{RSZ24}.

\vspace{1em}
\noindent\textbf{Example 5.1. (stochastic 3D tamed Navier-Stokes equation with reflection)}
Let $\mathbb{T}^3=[0,2\pi)^2$ be the torus in $\mathbb{R}^3$.  The stochastic 3D tamed Navier-Stokes equations are as follows:
\begin{equation}\label{eq:tamedNSE}
\begin{cases}
	d u(t) = \big[\nu \Delta u(t) - (u(t) \cdot \nabla) u(t) - g_N(|u(t)|^2) u(t) - \nabla p(t)\big] dt + \sum_{k=1}^\infty\sigma_k(t, u(t)) dW^k(t)\\
	\qquad \quad + L(t,x) , \ \text{on } (0, T] \times \mathbb{T}^3,\\
	\text{div } u = 0,\\
	u(0, x) = u_0(x),
\end{cases}
\end{equation}
where $u: [0, T] \times\mathbb{T}^3\rightarrow \mathbb{R}^3$ represents the velocity field, $p: [0, T] \times \mathbb{T}^3 \rightarrow \mathbb{R}$ is the pressure, $\nu > 0$ is the viscosity coefficient, $\{W_k,k\geq 1\}$ is a sequence of independent Brownian motions on a complete filtered probability space $(\Omega,\mathcal{F},\mathbb{F},\P)$, and $W=\{W_k\}$ can be viewed as a cylindrical Wiener process on the Hilbert space $l_2$ (space of all sequences of square summable real numbers with standard norm $\|\cdot\|_{l_2}$), the coefficient $\sigma(t,x,u): [0,T]\times \mathbb{T}^3\times \mathbb{R}^3 \to \mathbb{R}^3\times l_2$, and $g_N: \mathbb{R}_+ \rightarrow \mathbb{R}_+$ is a smooth taming function satisfying for some given $N > 0$:
\[
\begin{cases}
g_N(r) = 0, & \text{if } r \leq N, \\
g_N(r) = \dfrac{r - N}{\nu}, & \text{if } r \geq N + 1, \\
0 \leq g_N'(r) \leq 2/(\nu\wedge 1), & r \geq 0.
\end{cases}
\]
This equation was studied by M. R\"{o}ckner and X. Zhang \cite{RZ09}, also  M. R\"{o}ckner and T. Zhang \cite{RZ12}, etc. The motivation to study this equation originates from the deterministic case, where a bounded strong solution of the classical 3D Navier-Stokes equation coincides with the solution of equation \eqref{eq:tamedNSE} (with $\sigma = 0$) for large enough $N$.

We introduce $\mathbb{H}^m=\{u\in H^m(\mathbb{T}^3;\mathbb{R}^3):\text{div}u=0,\int_{\mathbb{T}^3}u=0\}$ with the usual $H^m$ norm, $m\in \N$. Note that $\mathbb{H}^0$ coincides with  $\{u\in L^2(\mathbb{T}^3;\mathbb{R}^3):\text{div} u=0,\int_{\mathbb{T}^3}u=0\}.$ Let $P$ be the Leray orthogonal projection from $L^{2}(\mathbb{T}^3;\R^{3})$ to $\mathbb{H}^{0}$ (see, e.g., \cite{RRS16} Chapter 2). It is well known that $P$ commutes with the derivative operators and is symmetric. For any $u \in \mathbb{H}^{2}$, define
\begin{equation*}
A(u) := P \Delta u-P ((u \cdot \nabla)u) - P (g_{N}(|u|^{2}) u),\quad  B_k(t, u) :=P \sigma_k(t,u).
\end{equation*}
Then $B(t,u)$ is defined by 
$$B(t,u)a=\sum_k \sigma_k(t,u)a_k, \quad a=(a_k)\in l_2.$$
Assume that $B:[0,T]\times \mathbb{H}^i\to L_2(l_2,\mathbb{H}^i)$, $i=1,2$ for the well-definedness of stochastic integral. 
With Leray projection, we shall consider the following equivalent abstract stochastic evolution equation with reflection in an infinite-dimensional ball:
\begin{equation}\label{eq:tamedRNSE}
\begin{cases}
	\mathrm{d} u(t) =A(u(t))dt +B(t, u(t)) dW(t)+dL(t),\quad t\in(0,T]\\
	u(0) = u_{0} \in \overline{D}_{\mathbb{H}^{1}}.
\end{cases}
\end{equation}
where $L(t)$ is an $\mathbb{H}^1$-valued adapted process of locally bounded variation, and denote by $\overline{D}_{\mathbb{H}^{1}}$ the closed unit ball in $\mathbb{H}^{1}$.

Set $H := \mathbb{H}^1$ with inner product $\langle u,v\rangle_{\mathbb{H}^1}=\langle u,v\rangle_{\mathbb{H}^0}+\langle \nabla u,\nabla  v\rangle_{L^2}, \text{ for }u,v\in \mathbb{H}^1,$ and $V :=\mathbb{H}^2$. Let 
$$\langle u,v\rangle_{\mathbb{H}^2,\mathbb{H}^0}:=\langle u,v\rangle_{\mathbb{H}^0}-\langle u,\Delta v\rangle_{L^2}, \quad \text{ for } u\in\mathbb{H}^0, v\in \mathbb{H}^2.$$
Then we have the Gelfand triple $\mathbb{H}^2\subseteq \mathbb{H}^1\subseteq \mathbb{H}^0$ and the embeddings are compact.  It is easy to see that condition \hypref{H1} are satisfied. We assume that the noise coefficient $\sigma_k(t, u)$ satisfies 
\begin{align}
\sum_{k=1}^{\infty} \|\sigma_k(u) - \sigma_k(v)\|_{\mathbb{H}^1}^2 &\leq c \|u - v\|_{\mathbb{H}^1}^2, \text{ for } u, v \in \mathbb{H}^1 \label{sigma1}\\
\sum_{k=1}^{\infty} \|\sigma_k(u)\|_{\mathbb{H}^1}^2 &\leq c(1+\|u\|_{\mathbb{H}^1}^2), \text{ for } u \in \mathbb{H}^1.\label{sigma2}
\end{align}
Then conditions in \hypref{H5} in Section 2 are satisfied. From the estimates (3.5)-(3.7) in \cite{RZ12} or Lemma 2.3 in \cite{RZ09}, one sees that \hypref{H3} is satisfied with $\alpha=2$. 
The $\mathbb{H}^0$ part of  \hypref{H2} follows from the estimate (3.20) in \cite{RZ12}:
\begin{equation*}
\langle A(u)-A(v),u-v\rangle_{\mathbb{H}^0}\leq-\frac{1}{2}\|u-v\|_{\mathbb{H}^1}^2+C(\|v\|_{\mathbb{H}^1}\|v\|_{\mathbb{H}^2}+1)\|u-v\|_{\mathbb{H}^0}^2, \quad \forall u, v \in \mathbb{H}^2.
\end{equation*}
We estimate the other part as follows: for $u,v\in\mathbb{H}^2$,
\begin{equation*}
-\langle A(u)-A(v),\Delta (u-v)\rangle_{L^2}=A_1+A_2+A_3.
\end{equation*}
where \begin{equation*}
\begin{aligned}
	A_1&=-\nu \langle P\Delta(u-v),\Delta(u-v)\rangle_{\mathbb{H}^0}=-\nu\|\Delta(u-v)\|_{\mathbb{H}^0}^2\lesssim -\|u-v\|_{\mathbb{H}^2}^2,\\
	A_2
	&=\langle u\cdot\nabla (u-v)+(u-v)\cdot \nabla v, \Delta(u-v)\rangle_{\mathbb{H}^0}\\
	&\leq \|u\|_{L^\infty}\|\nabla (u-v)\|_{L^2}\|\Delta(u-v)\|_{L^2}+\|u-v\|_{L^6}\|\nabla v\|_{L^3}\|\Delta(u-v)\|_{L^2}\\
	&\lesssim (\|u\|_{\mathbb{H}^2}+\| v\|_{\mathbb{H}^2})\|u-v\|_{\mathbb{H}^1}\|u-v\|_{\mathbb{H}^2}\\
	&\leq \frac{1}{4}\|u-v\|_{\mathbb{H}^2}^2+C (\|u\|_{\mathbb{H}^2}^2+\| v\|_{\mathbb{H}^2}^2)\|u-v\|_{\mathbb{H}^1}^2,
\end{aligned}
\end{equation*}
where H\"{o}lder's  inequality and Sobolev embeddings $H^2\subseteq L^\infty$, $H^1\subseteq L^6$, $H^{1/2}\subseteq L^3$ (see, e.g.,  \cite{RRS16} Theorem 1.18) and Young's inequality are used.  And $|g_N(r)|\leq r+N$, $|g_N'(r)|\leq 2$ give that
\begin{equation*}
\begin{aligned}
	A_3&=\langle g_N(|u|^2)u-g_N(|v|^2)v,\Delta(u-v)\rangle_{\mathbb{H}^0}\\
	&=\langle g_N(|u|^2)(u-v)+[g_N(|u|^2)-g_N(|v|^2)]v,\Delta(u-v)\rangle_{\mathbb{H}^0}\\
	&\leq \langle (|u|^2+N)(u-v)+2|u+v||u-v|v,\Delta(u-v)\rangle_{\mathbb{H}^0}\\
	&\leq (\|u\|_{L^6}^2+\|v\|_{L^6}^2)\|u-v\|_{L^6}\|\Delta(u-v)\|_{L^2}+N\|u-v\|_{L^2}\|\Delta(u-v)\|_{L^2}\\
	&\lesssim  (\|u\|_{\mathbb{H}^1}^2+\|v\|_{\mathbb{H}^1}^2)\|u-v\|_{\mathbb{H}^1}\|u-v\|_{\mathbb{H}^2}+N\|u-v\|_{\mathbb{H}^0}\|u-v\|_{\mathbb{H}^2}\\
	&\leq \frac{1}{4}\|u-v\|_{\mathbb{H}^2}^2+C(N+\|u\|_{\mathbb{H}^1}^4+\|v\|_{\mathbb{H}^1}^4)\|u-v\|_{\mathbb{H}^1}^2
\end{aligned}
\end{equation*}
Thus $\rho(u)$ and $\eta(v)$ are functions with growth of order 2 in $V$-norm and growth of order 4 in  $H$-norm, verifying condition \hypref{H2}. Condition \hypref{H4} is satisfied with $\alpha=2,\beta=6$:
\begin{equation*}
\begin{aligned}
	\|\nu P\Delta u\|_{\mathbb{H}^0}^2&\leq  \nu ^2\|\Delta u\|_{L^2}\leq \nu ^2\|u\|_{\mathbb{H}^2},\\
	\|P(u\cdot\nabla) u\|_{\mathbb{H}^0}^2&\leq \|(u\cdot\nabla) u\|_{L^2}^2\leq \|u\|_{L^6}^2 \|\nabla u\|_{L^3}^2\leq \|u\|_{\mathbb{H}^1}^2 \| u\|_{\mathbb{H}^2}^2,\\
	\|P(g_N(|u|^2)u\|_{\mathbb{H}^0}^2&\leq \|(|u|^2+N)u\|_{L^2}^2\leq\|u\|_{L^6}^6+ N\|u\|_{L^2}^2 \leq\| u\|_{\mathbb{H}^1}^6+ N\|u\|_{\mathbb{H}^0}^2.
\end{aligned} 
\end{equation*}
Therefore, by Theorem \ref{result}, we have established the well-posedness of solutions to the stochastic 3D tamed Navier-Stokes equations with reflection. Remark that the major difference compared with the classical stochastic Navier-Stokes equations is that in the tamed version, the taming term $g_N(|u|^2)u$ provides additional dissipation that compensates for the lack of monotonicity in the convective term.

\vspace{1em}
\noindent\textbf{Example 5.2. (Reflected quasilinear SPDE)}
Let $\mathcal{O}$ be a bounded domain in $\mathbb{R}^d$ with smooth boundary. Consider the reflected quasilinear SPDE:
\begin{equation}\label{eq:quasilinear}
\begin{cases}
	d u(t, x) = [\nabla \cdot a(t, x, u(t,x), \nabla u(t,x)) - a_0(t, x, u(t,x), \nabla u(t,x))] dt \\
	\quad \quad \quad \quad + \sum_k\sigma_k(u(t,x)) dW^k(t) + dL(t, x), \quad t \geq 0, x \in \mathcal{O}, \\
	u(0, x) = u_0(x), \quad x \in \mathcal{O}, \\
	u(t, x) = 0, \quad t > 0, x \in \partial \mathcal{O}.
\end{cases}
\end{equation}
Here, $H = L^2(\mathcal{O})$, $V = W_0^{1,\alpha}(\mathcal{O})$ for  $\alpha>1$ if $d = 1, 2$, and $\alpha\geq \frac{2d}{d+2}$ if $d\geq 3$. $W=\{W^k\}_{k=1}^\infty$ is a sequence of independent Brownian motions, and $L$ is the reflection process which is an $L^2(\mathcal{O})$-valued process of bounded variation. The functions $a_i$, $i=0,...,d$, satisfy conditions (S1)-(S3) and (S4)' as in \cite{RSZ24}, and the mapping $\sigma(\cdot)= (\sigma_k(\cdot))_{k \in \mathbb{N}} :L^2(\mathcal{O})\to L_2(l_2,L^2(\mathcal{O}))$ defined by
\begin{equation*}
\sigma(u)a := \sum_{k=1}^\infty \sigma_k(u) a_k, \quad a = (a_k)_{k \in \mathbb{N}} \in l_2, \quad u\in L^2(\mathcal{O}),
\end{equation*}
is Lipschitz and of linear growth (similar to \eqref{sigma1} and \eqref{sigma2}). The operator $A$ is defined by $\langle A(u), v \rangle = -\int_{\mathcal{O}} \left\{ \sum_{i=1}^d a_i(x, u, \nabla u) \partial_i v + a_0(x, u, \nabla u) v \right\} dx$. These conditions imply that the assumptions of Theorem \ref{result} hold, as shown in \cite{RSZ24}. Consequently the corresponding reflected SPDE admits a unique solution $(u, L)$ in the sense of Definition \ref{solutiondef}. A typical example if \eqref{eq:quasilinear} is the $p$-Laplacian for $p \geq 2$,
\begin{equation}\label{plaplace}
\partial_t u = \nabla \cdot \left( |\nabla u|^{p-2} \nabla u \right) - c |u|^{p-2} u, 
\end{equation}
where $c \geq 0$. In this case we take $\alpha = p$, and it is easy to verify that (S1)-(S4)' are satisfied.

\vspace{1em}
\noindent\textbf{Example 5.3. (Reflected Cahn-Hilliard Equation)}
Let $\mathcal{O} \subset \mathbb{R}^d$ ($d=1,2,3$) be a bounded domain with smooth boundary, $\nu$ be the outward unit normal vector on $\partial O$. The reflected stochastic Cahn-Hilliard equation is:
\begin{equation*}
\begin{cases}
	d u(t) = -\Delta^2 u(t) dt + \Delta \varphi(u(t)) dt + \sigma(u(t)) dW(t) + dL(t), \quad t \geq 0, \\
	\nabla u \cdot \nu = \nabla (\Delta u) \cdot \nu = 0 \text{ on } \partial \mathcal{O}, \\
	u(0) = u_0 \in \overline{D} \subseteq L^2(\mathcal{O}),
\end{cases}
\end{equation*}
Here, $H = L^2(\mathcal{O})$, $V = \{ u \in H^2(\mathcal{O}) : \nabla u \cdot \nu = \nabla(\Delta u) \cdot \nu = 0 \text{ on } \partial \mathcal{O} \}$, $\overline{D}$ is the closed unit ball in $H$. The embedding $V\subseteq H$ is compact. And $W$ is cylindrical Wierner process on a separable Hilbert space $U$, and $L$ is an adapted $H$-valued process of bounded variation. We assume that the nonlinear term $\varphi$ satisfies the following conditions: $\varphi \in C^1(\mathbb{R}, \mathbb{R})$ and there exist constants $C \geq 0$ and $2 \leq p \leq \frac{d+4}{d}$ such that for any $x, y \in \mathbb{R}$, $\varphi'(x) \geq -C$, $|\varphi(x)| \leq C(1 + |x|^p)$ and 
\[
|\varphi(x) - \varphi(y)| \leq C(1 + |x|^{p-1} + |y|^{p-1})|x - y|.
\]
The operator $A$ is given by $A(u) = -\Delta^2 u + \Delta \varphi(u)$. In \cite{RSZ24}, the conditions associated with $A$ are verified. If $\sigma: H \to L_2(U,H)$ is Lipschitz and of linear growth, then the associated reflected SPDE has a unique solution by Theorem \ref{result}.

\vspace{1em}
\noindent\textbf{Example 5.4. (Reflected 2D liquid crystal model)}
The reflected stochastic 2D Liquid Crystal model (simplified Ericksen-Leslie system) is:
\begin{equation*}
\begin{cases}
	\partial_t u = \Delta u - (u \cdot \nabla) u - \nabla p - \nabla \cdot (\nabla n \otimes \nabla n) + \sigma_1(u, n) dW(t) + dL_u(t), \\
	\partial_t n = \Delta n - (u \cdot \nabla) n - \Phi(n) + \sigma_2(u, n) dW(t) + dL_n(t), \\
	\nabla \cdot u = 0, \\
	u|_{\partial\mathcal{O}} = 0, \quad \frac{\partial n}{\partial \nu}|_{\partial\mathcal{O}} = 0, \\
	u(0) = u_0, \quad n(0) = n_0, \quad (u_0, n_0) \in \overline{D} \subseteq H \times H^1(\mathcal{O})^3.
\end{cases}
\end{equation*}
where $\mathcal{O}$ is a bounded domain in $\mathbb{R}^2$ with smooth boundary $\partial \mathcal{O}$, $u : [0, T] \times \mathcal{O} \to \mathbb{R}^2$ is the velocity, $p : [0, T] \times \mathcal{O} \to \mathbb{R}$ is the pressure, $n : [0, T] \times \mathcal{O} \to \mathbb{R}^3$ is the director field of liquid crystal molecules, $\nu$ is the outward unit normal vector on $\partial \mathcal{O}$. By the symbol $\nabla n \otimes \nabla n$ we mean a $2 \times 2$ matrix with entries defined by
\[
(\nabla n \otimes \nabla n)_{i,j} = \sum_{k=1}^{3} (\partial_i n_k)(\partial_j n_k),
\]
where $\partial_i$ denotes the partial derivative with respect to $x_i$ for $i = 1, 2$. We assume that $\Phi : \mathbb{R}^3 \to \mathbb{R}^3$ satisfies the following conditions: there exists a $k$-th polynomial $\varphi : [0, \infty) \to \mathbb{R}$ for some $k \in \mathbb{N}$ such that
\[
\Phi(n) = \varphi(|n|^2)n = \left( \sum_{i=0}^{k} a_i |n|^{2i} \right) n,
\]
where $a_i \in \mathbb{R}$ for $i = 0, 1, \ldots, k-1$ and $a_k > 0$.
Let $V=\{u\in H^1(\mathcal{O})^2:\nabla\cdot u=0,\,u|_{\partial\mathcal{O}}=0\}$. 

Denote by $H$ the closure of $V$ under the $L^2$-norm $\|u\|_H^2:=\int_\mathcal{O}|u(x)|^2dx$. 
Set
\[
\mathbb{H}:=H\times [H^1(\mathcal{O})^3],\quad \mathbb{V}:=V\times\Big\{n\in H^2(\mathcal{O})^3:\frac{\partial n}{\partial \nu}=0\Big\},
\]
with norms in $H$ and $V$ denoted by
\[
\|X\|_\mathbb{H}^2:=\|u\|_H^2+\|n\|_{H^1}^2,\quad \|X\|_\mathbb{V}^2:=\|u\|_V^2+\|n\|_{H^2}^2
\]
for $X=(u,n)$. Then we have the Gelfand triple $\mathbb{V}\subseteq\mathbb{H}\subseteq\mathbb{V}^*$ and the embedding $\mathbb{V}\subseteq\mathbb{H}$ is compact. Denote by $\overline{D}$ the closed unit ball in $\mathbb{H}$. The $\mathbb{H}$-valued reflection process is $L(t) = (L_u(t), L_n(t))$, and $W$ is cylindrical Wiener process on a separable Hilbert space $U$.

Note that
\[
\nabla\cdot(\nabla n\otimes\nabla n)=\frac{1}{2}\nabla(|\nabla n|^2)+\nabla n\cdot\Delta n.
\]
Let $P_H:L^2(\mathcal{O})^2\to H$ be the Helmholtz-Leray projection. Set
\[
A(X):=\begin{pmatrix} P_H[\Delta u-(u\cdot\nabla)u-\nabla n\cdot\Delta n]\\ \Delta n-(u\cdot\nabla)n-\varphi(n) \end{pmatrix}.
\]
The assumptions \hypref{H1}-\hypref{H4} regarding the operator $A$ are verified in \cite{RSZ24}. If $\sigma = (\sigma_1, \sigma_2): \mathbb{H} \to L_2(U,\mathbb{H})$ is Lipschitz and of linear growth, then the reflected system admits a unique solution $((u, n), (L_u, L_n) )$ by Theorem \ref{result}. This framework also applies to the reflected Allen-Cahn-Navier-Stokes system and reflected magneto-hydrodynamic (MHD) equations.

\vspace{1em}
\noindent\textbf{Example 5.5. (Reflected 3D Leray-$\alpha$ model with fractional dissipation)}
We consider the reflected 3D Leray-$\alpha$ model with fractional dissipation on the 3D torus $\mathbb{T}^3=[0,2\pi]^3$ with periodic boundary conditions:
\begin{equation*}
\begin{cases}
	du +[\nu(-\Delta)^{\theta_2}u+(v\cdot\nabla)u+\nabla p]dt=g(u)dW(t) + dL(t),\\
	u=v+\alpha^{2\theta_1}(-\Delta)^{\theta_1}v,\\
	\nabla\cdot u=0,\quad \nabla\cdot v=0,\\
	\int_{\mathbb{T}^3}u(x)dx=0,\quad \int_{\mathbb{T}^3}v(x)dx=0,
\end{cases}
\end{equation*}
where $W(t)$ is a cylindrical Wiener process on a separable Hilbert space $U$. In particular, in the case of $\theta_1=0$, the above model becomes the hyperviscous Navier-Stokes equations and it is well known that this system has a unique global solution for $\theta_2\geq\frac{5}{4}$ without reflection. Since we work with periodic boundary condition, we can expand the velocity in Fourier series as
\begin{equation*}
u(x)=\sum_{k\in \mathbb{Z}_0^3}\hat{u}_k e^{ik\cdot x}, \quad \text{with  } \hat{u}_k\in \mathbb{C}^3,\hat{u}_{-k}=\hat{u}_k^\ast \text{ for every } k,
\end{equation*}
where $\mathbb{Z}_0^3=\mathbb{Z}^3\backslash\{0\}$ and $\hat{u}_k^\ast$ denotes the complex conjugate of $\hat{u}_k$.
For $s\in \mathbb{R}$, we define  the divergence free Sobolev space by
\begin{equation*}
\mathbb{H}^s:=\{u:\|u\|^2_s=\sum_{k\in \mathbb{Z}_0^3}|k|^{2s}|\hat{u}_k|^2<\infty\text{ and }\hat{u}_k\cdot k=0 \text{ for every } k\},
\end{equation*}
which is a Hilbert space with scalar product
\begin{equation*}
\langle u ,v\rangle_{\mathbb{H}^s}=\sum_{k\in \mathbb{Z}_0^3}|k|^{2s}\hat{u}_k\cdot \hat{v}_{-k}.
\end{equation*}
For $\theta_1\geq 0,\theta_2>\frac{1}{2},\theta_1+\theta_2\geq\frac{5}{4}$ with the initial data in $\mathbb{H}^0$, we consider the following Gelfand triple:
\begin{equation*}
\mathbb{H}^{\theta_2}\subset\mathbb{H}^0\subset \mathbb{H}^{-\theta_2}.
\end{equation*}
If we assume that $g$ is a global Lipschitz mapping from $\mathbb{H}^0$ to $L_2(U,\mathbb{H}^0)$, then \hypref{H1}-\hypref{H5} are satisfied. For further details, we refer the reader to the proof of Theorem 3.4 in \cite{Li23}, where Hypothesis 3.1 is replaced by Lipschitz conditions. Consequently, it follows from Theorem 4.1 that he reflected system admits a unique solution $(u,L).$

\section*{Appendix} 
\renewcommand{\theequation}{A.\arabic{equation}} 

In this section, we give the proofs of Lemmas \ref{lem3.4}-\ref{lem3.8}.

\vspace{1em}
\noindent{\textit{ Proof of Lemma \ref{lem3.4}.}} 
Fix $n\in\mathbb{N}$. Let $\psi(z) = |z|_H^4$, $z\in H$. Knowing that
\[
\nabla \psi(z) = 4|z|_H^2 z \quad\text{and}\quad  D^2\psi(z) = 8z\otimes z + 4|z|_H^2 I_H, \quad z\in H,
\]
and applying Itô's formula,  we have for $t\in[0,T]$,
\begin{equation}\label{lem3.4-3}
	\begin{aligned}
		|X^n(t)|_H^4 &= |X_0|_H^4 + 2\int_0^t |X^n(s)|_H^2 \left[2\langle X^n(s),A(s, X^n(s))\rangle +\|B(s,X^n(s))\|_{L_2}^2 \right]\, ds\\
		&\quad + 4\int_0^t |X^n(s)|_H^2 \langle X^n(s), B(s,X^n(s))\,dW(s) \rangle\\
		&\quad - 4n\int_0^t |X^n(s)|_H^2 \langle X^n(s), X^n(s) - \pi(X^n(s)) \rangle \, ds\\
		&\quad + 4\int_0^t \langle X^n(s), B(s,X^n(s)) \rangle^2.
	\end{aligned}
\end{equation}

Observe that by \eqref{pi2} we have $\langle X^n(s), X^n(s) - \pi(X^n(s)) \rangle \geq 0$ for all $s\in[0,T]$. By \hypref{H3} and \hypref{H5} respectively, the second and the last terms on RHS of \eqref{lem3.4-3} are dominated by
\begin{align}
	&2\int_0^t |X^n(s)|^2_H\left[C_0(1+|X^n(s)|_H^2)-c\|X^n(s)\|_V^\alpha\right] \,ds,\nonumber\\
	&4\int_0^t|X^n(s)|_H^2 \|B(s,X^n(s)\|_{L^2}^2ds\leq C\int_0^t|X^n(s)|^2_H(1+|X^n(s)|^2_H)ds,\nonumber
\end{align}
repectively. Rearranging the terms and taking expectation of \eqref{lem3.4-3} yields that
\begin{align}\label{lem3.4-5}
	&\mathbb{E}\left[ \sup_{0\leq r\leq t} |X^n(r)|_H^4 \right] + 4n\,\mathbb{E}\left[ \int_0^t |X^n(s)|_H^2 \langle X^n(s), X^n(s) - \pi(X^n(s)) \rangle \, ds \right]\nonumber \\
	&\quad \leq |X_0|_H^4 + 4\,\mathbb{E}\left[ \sup_{0\leq r\leq t} \left| \int_0^r |X^n(s)|_H^2 \langle X^n(s), B(s,X^n(s))dW(s)  \rangle \, \right| \right]\nonumber\\
	&\quad+C\,\E \left[\int_0^t|X^n(s)|^2_H(1+|X^n(s)|^2_H)\, ds\right].
\end{align}

Let $\{e_k\}$ be an orthonormal basis of the Hilbert space $U$. Applying Burkholder's inequality, Young's inequality and using \hypref{H5} yield
\begin{align}
	&\mathbb{E}\left[ \sup_{0\leq r\leq t} \left| \int_0^r |X^n(s)|_H^2 \langle X^n(s), B(s,X^n(s))dW(s)  \rangle \right| \right]\nonumber \\
	&\quad\leq C\,\mathbb{E}\left[ \left( \int_0^t |X^n(s)|_H^4 \sum_{k=1}^\infty\langle X^n(s), B(s,X^n(s))e_k \rangle^2 \, ds \right)^{\frac{1}{2}} \right]\nonumber \\
	&\quad\leq C\,\mathbb{E}\left[ \left( \sup_{0\leq r\leq t} |X^n(r)|_H^2 \right) \left( \int_0^t \sum_{k=1}^\infty\langle X^n(s), B(s,X^n(s))e_k \rangle^2 \, ds \right)^{\frac{1}{2}} \right]\nonumber \\
	&\quad\leq \frac{1}{2}\mathbb{E}\left[ \sup_{0\leq r\leq t} |X^n(r)|_H^4 \right] + C\,\mathbb{E}\left[ \int_0^t |X^n(s)|_H^2 \|B(s,X^n(s))\|_{L^2}^2\, ds \right]\nonumber\\
	&\quad\leq \frac{1}{2}\mathbb{E}\left[ \sup_{0\leq r\leq t} |X^n(r)|_H^4 \right] + C\,\mathbb{E}\left[ \int_0^t |X^n(s)|_H^2 (1+|X^n(s))|_H^2)\, ds \right].\label{lem3.4-6}
\end{align}
Substituting \eqref{lem3.4-6} into \eqref{lem3.4-5} we obtain that
\begin{equation}\label{lem3.4-7}
	\begin{aligned}
		&\mathbb{E}\left[ \sup_{0\leq r\leq t} |X^n(r)|_H^4 \right] + 4n\mathbb{E}\left[ \int_0^t |X^n(s)|_H^2 \langle X^n(s), X^n(s) - \pi(X^n(s)) \rangle \, ds \right] \\
		&\leq C\,|X_0|_H^4 + C\,\mathbb{E}\left[ \int_0^t \left(1 + |X^n(s)|_H^4 \right) \, ds \right].
	\end{aligned}
\end{equation}
Applying Gronwall's lemma to \eqref{lem3.4-7} implies \eqref{lem3.4-1}. Finally, the combination of \eqref{lem3.4-7} and \eqref{lem3.4-1} implies \eqref{lem3.4-2}. The proof of Lemma \ref{lem3.4} is complete.\qedd

\vspace{1em}
\noindent{\textit{ Proof of Lemma \ref{lem3.6}.}} 
By It\^{o}'s formula, we have  for $t\in[0,T]$,
\[
\begin{aligned}
	\left|X^{n}(t)\right|_{H}^{2}&=\left|X_0 \right|_{H}^{2}+2\int_{0}^{t}\left(\left\langle X^{n}(s),A(s,X^{n}(s))\right\rangle+\left\|B(s,X^{n}(s)\right\|_{L^2}^{2}\right) ds\\
	&\quad+2\int_{0}^{t}\langle X^{n}(s),B(s,X^{n}(s))dW(s)\rangle \\
	&\quad-2n\int_{0}^{t}\left\langle X^{n}(s),X^{n}(s)-\pi\left(X^{n}(s)\right)\right\rangle ds.
\end{aligned}
\]
In view of \eqref{pi2},
\begin{equation*}
	\begin{aligned}
		&2n\int_{0}^{t}\left\langle X^{n}(s),X^{n}(s)-\pi\left(X^{n}(s)\right)\right\rangle ds\\
		=&\,2n\int_{0}^{t}\left|X^{n}(s)-\pi\left(X^{n}(s)\right)\right|_{H}^{2}ds+2n\int_{0}^{t}\left\langle\pi\left(X^{n}(s)\right),X^{n}(s)-\pi\left(X^{n}(s)\right)\right\rangle ds \\
		\leq &\, 2n\int_{0}^{t}\left|X^{n}(s)-\pi\left(X^{n}(s)\right)\right|_{H}^{2}ds+2n\int_{0}^{t}\left|X^{n}(s)-\pi\left(X^{n}(s)\right)\right|_{H}ds,\quad t\in[0,T].
	\end{aligned}
\end{equation*}

By Lemma \ref{lem3.4}, \hypref{H3}, \hypref{H5} and using Burkholder's inequality analogous to \eqref{lem3.4-6} we arrive at
\begin{align*}
	\sup_{n}\mathbb{E}\left[\left(n\int_{0}^{T}\left|X^{n}(s)-\pi\left(X^{n}(s)\right)\right|_{H}ds\right)^{2}\right]\leq C+C\sup_{n}\mathbb{E}\left[\sup_{t\in[0,T]}\left|X^{n}(t)\right|_{H}^{4}\right]&\leq M_{1}(T)
\end{align*}
and \eqref{lem3.6-2}, \eqref{lem3.6-3} also follow.
Hence the proof of Lemma \ref{lem3.6} is complete.\qedd

\vspace{1em}
\noindent{\textit{Proof of Lemma \ref{lem3.8}.} }
Define two functions $G,g:H\to[0,\infty)$ by $
G(y)=|y-\pi(y)|^{4}, g(y)=|y-\pi(y)|^{2}$ for $y\in H$.
Then for $y,u,v\in H$,
\begin{align*}
	\nabla G(y)(v)&=4g(y)\langle y-\pi(y),v\rangle,\nonumber\\
	D^2G(y)(u,v)&=8\langle y-\pi(y),u\rangle\langle y-\pi(y),v\rangle\nonumber\\
	&\quad+4g(y)1_{|y|_{H}>1}\left[\langle u,v\rangle\left(1-\frac{1}{|y|_{H}}\right)+\frac{1}{|y|_{H}^{3}}\langle y,u\rangle\langle y,v\rangle\right].
\end{align*}
Applying It\^{o}'s formula of $G$, with $G(X_0)=0$, we have
\begin{equation}
	\begin{aligned}
		G(X^{n}(t))&=4\int_{0}^{t}g\left(X^{n}(s)\right)\left\langle X^{n}(s)-\pi\left(X^{n}(s)\right),A(s,X^{n}(s))\right\rangle ds\\
		&\quad+4\int_{0}^{t}g\left(X^{n}(s)\right)\left\langle X^{n}(s)-\pi\left(X^{n}(s)\right),B\left(s,X^{n}(s)\right)dW(s)\right\rangle \\
		&\quad-4n\int_{0}^{t}g\left(X^{n}(s)\right)\left| X^{n}(s)-\pi\left(X^{n}(s)\right)\right|_{H}^{2}ds\\
		&\quad+4\int_{0}^{t}\left\langle X^{n}(s)-\pi\left(X^{n}(s)\right),B\left(s,X^{n}(s)\right)dW(s)\right\rangle^{2}\\
		&\quad+2\int_{0}^{t}g\left(X^{n}(s)\right)1_{\left|X^{n}(s)\right|>1}\Bigg[\left\|B\left(s,X^{n}(s)\right)\right\|_{L^2}^{2}\left(1-\frac{1}{\left|X^{n}(s)\right|_{H}}\right)ds\\
		&\quad\quad+\frac{1}{\left|X^{n}(s)\right|_{H}^{3}}\left\langle X^{n}(s),B\left(s,X^{n}(s)\right)dW(s)\right\rangle^{2}\Bigg]\\
		&:=I_{1}^{n}(t)+I_{2}^{n}(t)+I_{3}^{n}(t)+I_{4}^{n}(t)+I_{5}^{n}(t),\quad t\in[0,T].
	\end{aligned}
\end{equation}
From \hypref{H3} and \hypref{H5} respectively, we have for $t\in[0,T]$,
\begin{align*}
	I_{1}^{n}(t)&=4\int_{0}^{t}g(X^{n}(s))\lambda(|X^{n}(s)|_H)\left\langle X^{n}(s),A(s,X^{n}(s))\right\rangle\, ds\\
	&\leq \int_{0}^{t}g(X^{n}(s))\left[C_0(1+|X^{n}(s)|_H^2-c\|X^{n}(s)\|_V^\alpha\right]\,ds.\\
	I_{4}^{n}(t)&\leq 4\int_{0}^{t}\left|X^{n}-\pi(X^n(s))\right|_H^2 \|B(s,X^n(s))\|_{L^2}^2\,ds\leq C\int_{0}^{t} \left|X^{n}-\pi(X^n(s))\right|_H^2(1+|X^n(s)|_{H}^2)\,ds.
\end{align*}
Here the function $\lambda: [0,\infty) \to [0,1]$  is given by
\[
\lambda(r) =
\begin{cases}
	0, & \text{if } 0\leq r\leq 1, \\
	1 - \frac{1}{r}, & \text{if } r > 1.
\end{cases}
\]
Then $y - \pi(y) = \lambda(|y|_H)y$, for $y \in H$. For the term $I_{5}^{n}$, notice that $1-1/|X^n(s)|_H<1$ if $|X^n(s)|_H>1$. We have
\begin{align*}
	I_{5}^{n}(t)\leq 4\int_{0}^{t} g(X^{n}(s))\|B(s,X^n(s))\|_{L^2}^2 \, ds\leq C\int_0^t g(X^{n}(s))(1+|X^n(s))|_H^2) ds.
\end{align*}
Using Burkholder's inequality analogous to \eqref{lem3.4-6} we have for $t\in[0,T]$,
\begin{align*}
	\mathbb{E}\left[\sup_{0\leq s\leq t}\left|I_{2}^{n}(s)\right|\right]
	&\leq\frac{1}{2}\mathbb{E}\left[\sup_{0\leq s\leq t}G\left(X^{n}(s)\right)\right]+C\,\mathbb{E}\left[\int_{0}^{t}\left|X^{n}(s)-\pi\left(X^{n}(s)\right)\right|_{H}^{2}\|B(s,X^n(s))\|_{L^2}^2\, ds\right].
\end{align*}

Hence, from the above estimates it follows that, for $t\in[0,T]$,
\[
\mathbb{E}\left[\sup_{0\leq s\leq t}G\left(X^{n}(s)\right)\right]\leq C\,\mathbb{E}\left[\int_{0}^{t}\left(1+\left|X^{n}(s)\right|_{H}^{2}\right)\left|X^{n}(s)-\pi\left(X^{n}(s)\right)\right|_{H}^{2}ds\right].
\]
Thus, by \eqref{lem3.4-2} in Lemma \ref{lem3.4} and \eqref{lem3.6-2} in Lemma \ref{lem3.6}, we obtain that
\[
\lim_{n\rightarrow\infty}\mathbb{E}\left[\sup_{t\in[0,T]}\left|X^{n}(t)-\pi\left(X^{n}(t)\right)\right|_{H}^{4}\right]=0.
\]
This concludes the proof of Lemma \ref{lem3.8}.\qedd

\vspace{1em}
\noindent\textbf{Acknowledgement}
We thank Prof. Shijie Shang and Prof. Saisai Yang for helpful discussions.
This work is partly supported by National Key R$\&$D program of China (No. 2022 YFA1006001), and also the National Natural Science Foundation of China (NSFC) (No. 12526522, 12131019, 12371151, 12426655, 12571158).

\end{document}